\documentclass{amsart}
\theoremstyle{theorem}
\newtheorem{theorem}{Theorem}

\theoremstyle{definition}
\newtheorem*{definition}{Definition}
\newtheorem*{acknowledgment}{Acknowledgement}

\let\mkbld\boldsymbol
\def\killsphi#1{#1(\mkbld\phi)=0}
\def\Wr{\textup{Wr}}

\def\strut{\vrule height 12pt depth 4pt width 0pt}

\begin{document}

\title{Factorization of a Matrix Differential Operator Using Functions in its Kernel }

%\author{Alex Kasman}

\maketitle
\markboth{Factorization of a Matrix Differential Operator}{Factorization of a Matrix Differential Operator}

\begin{abstract}
Just as knowing some roots of a polynomial allows one to factor it, a
well-known result provides a factorization of any scalar differential
operator given a set of linearly independent functions in its kernel.
This note provides a straight-forward generalization to the case of
matrix coefficient differential operators.
\end{abstract}

\section{Motivation}

An ordinary differential operator (ODO) is a commonly used mathematical object that turns one function\footnote{Here and throughout this note, ``functions of $x$'' will be understood to mean ``\textit{sufficiently differentiable} functions of $x$'' even if differentiability is not specifically mentioned.} into another by adding up the products of some other specified functions with its derivatives.
Symbolically, an ODO is a polynomial in ``$\partial$'' having coefficients that are  functions of $x$.  The highest power of $\partial$ appearing with a non-zero coefficient is called the \textit{order} of the operator.  ODOs act on functions of $x$ according to the rule:
\begin{equation}
L=\sum_{i=0}^n \alpha_i(x)\partial^i\qquad
\hbox{implies}
\qquad
L(f)=\sum_{i=0}^n \alpha_i(x)f^{(i)}(x).\label{eqn:ODOapply}
\end{equation}
%In other words, $\partial^i$ acts by differentiating $i$ times with respect to $x$ and the coefficients act by multiplication.  

Much interest in differential operators comes from their use in writing linear differential equations.
However, 
ODOs also have the \textit{algebraic} structure of a \relax{non-commutative ring}.  
%Multiplication and factorization of ODOs has application in the study of certain special nonlinear partial %differential equations \cite{Darboux1,Darboux2} and in quantum physics \cite{QuantApp1,QuantApp2}.
They can be added as 
one would add any polynomials,
by combining the coefficients of similar powers of $\partial$, 
%polynomials in $\partial$
and
 multiplication is defined by extending the following rule for
the product of two monomials linearly over sums:
\begin{equation}
\left(\alpha(x)\partial^m\right)\circ\left(\beta(x)\partial^n\right)=\sum_{i=0}^m {m\choose i}\alpha(x)\beta^{(i)}(x)\partial^{m+n-i}.
\label{eqn:ODOmult}
\end{equation}
Although this definition may look complicated at first, in fact it is simply a consequence of the usual ``product rule'' from calculus, applied here to guarantee that multiplication corresponds to operator composition as one would expect.  That is, this definition was chosen so that $L\circ Q(f)=L(Q(f))$.
%Equivalently, for any function $f$ one has $L\circ Q(f)=L(Q(f))$, so multiplication corresponds to operator composition, as one would expect.

Given the rule \eqref{eqn:ODOmult} for multiplication, the inverse question of \textit{factorization} naturally arises.
One \textit{factorization method} for ODOs is surprisingly reminiscent of a familiar fact about polynomials.
If  $x=\lambda$ is a root of the polynomial $p(x)$, then you know that it has a factor of $x-\lambda$, the simplest first degree polynomial with this property.  
Similarly, for any non-zero function $f(x)$, $\partial-f'/f$ is the simplest first order differential operator having $f$ in its kernel, and if $L$ is any ODO with $f$ in its kernel then $L=Q\circ (\partial - f'/f)$ for some differential operator $Q$. 
Moreover, just as knowing additional roots of the polynomial $p$ would allow further factorization, one may factor a differential operator of order $n$ into the product of operators $n-k$ and $k$ from the knowledge of $k$ linearly independent functions in its kernel.  The general statement written in terms of Wronskian determinants\footnote{The Wronskian determinant $\Wr(\phi_1,\ldots,\phi_n)$ of the functions $\phi_i(x)$ is defined to be the determinant of the $n\times n$ matrix having $\frac{d^{i-1}}{dx^{i-1}}\phi_j(x)$ in row $i$ and column $j$.} is as follows:

\begin{theorem}[Scalar Case]\label{thm:scalar}
Let $\phi_1,\ldots,\phi_m$ be functions such that $\Wr(\phi_1,\ldots,\phi_m)\not=0$. Then (a) the unique monic differential operator $K$ of order $m$ satisfying $K(\phi_i)=0$ for $1\leq i\leq m$ is the one whose action on an arbitrary function $f(x)$ is given by the formula
%\begin{equation}
%K(f)=\frac{\Wr(\phi_1,\ldots,\phi_m,f)}{\Wr(\phi_1,\ldots,\phi_m)}\label{eqn:WrK}
%\end{equation}
$K(f)={\Wr(\phi_1,\ldots,\phi_m,f)}/{\Wr(\phi_1,\ldots,\phi_m)}$
and (b) if $L$ is any differential operator satisfying $L(\phi_i)=0$ for $1\leq i\leq m$ then there exists a differential operator $Q$ such that $L=Q\circ K$.
\end{theorem}

\section{The Matrix Case}

The purpose of this note is to generalize this well-known and useful result to the case of matrix coefficients\footnote{Another sort of higher-dimensional generalization is the case of \textit{partial} differential operators, which is considered in \cite{FactorPDOs}.}.  A matrix coefficient ordinary differential operator (MODO) is again a polynomial in $\partial$ with coefficients that depend on $x$, but we now consider the case in which those coefficients are $N\times N$ matrices.  The formulas for multiplying MODOs and for applying them to functions remain the same (see \eqref{eqn:ODOapply} and \eqref{eqn:ODOmult}), but now the products involving the coefficients are understood to be matrix products and the function $f$ is an $N$-vector valued function.

Matrix analogues of Theorem~\ref{thm:scalar}(a) already appear in the
literature.  For example, a result of Etingof-Gelfand-Retakh
\cite{QD} allows one to produce a monic MODO
with a specified kernel using quasi-determinants.  However, not only is there
no published analogue of Theorem~\ref{thm:scalar}(b) for MODOs, it
appears that many researchers have suspected that it does not
generalize nicely to the matrix case.  A common proof of
Theorem~\ref{thm:scalar}(b) depends on the fact that any non-zero 
ODO of order at most $n$ has a kernel of dimension at most $n$.  So, the fact that any
MODO with a singular leading coefficient has an infinite-dimensional
kernel is both an obstacle to generalizing the proof and reason
to doubt the validity of the equivalent statement for MODOs.

It is therefore good news that
Theorem~\ref{thm:modo}(b) below does indeed fully generalize the scalar result to the matrix case without imposing any additional restrictions on the leading coefficient, order or kernel of the operator $L$.  In addition to proving this \textit{new} result, what follows can be seen as providing a novel alternative proof to Theorem~\ref{thm:scalar} and Theorem~\ref{thm:modo}(a). 

\begin{definition}  Let  $\phi_1(x),\ldots,\phi_{MN}(x)$ be $N$-vector valued functions and let $\Phi$ be the $MN\times MN$ \textit{block Wronskian matrix}
%\footnote{In the block Wronskian matrix \eqref{eqn:blockWr}, the first $N$ rows are given by the selected vector functions and then every other element is the derivative of the element $N$ rows above it.  In the case $N=1$, the non-singularity of this matrix is famously equivalent to the linear independence of the functions, but that is not true in general.  For instance, if $N=2$, $M=1$, $\phi_1=(1\ 1)^{\top}$ and $\phi_2=(x\ x)^{\top}$, then $\det(\Phi)=0$ even though these are linearly independent functions of $x$.}
\begin{equation}
\Phi=\left(\begin{matrix} \phi_1&\phi_2&\cdots&\phi_{MN}\\
 \phi_1'&\phi_2'&\cdots&\phi_{MN}'\\
 \vdots&\vdots&\ddots&\vdots\\
  \phi_1^{(M-1)}&\phi_2^{(M-1)}&\cdots&\phi_{MN}^{(M-1)}
  \end{matrix}\right).
\label{eqn:blockWr}
\end{equation}
Note that the first $N$ rows of $\Phi$ are given by the selected vector functions and then every other element is the derivative of the element $N$ rows above it.  In the case $N=1$, the non-singularity of this matrix is famously equivalent to the linear independence of the functions, but that is not true in general.  For instance, if $N=2$, $M=1$, $\phi_1=(1\ 1)^{\top}$ and $\phi_2=(x\ x)^{\top}$, then $\det(\Phi)=0$ even though these are linearly independent functions of $x$.
\end{definition}

For the sake of brevity, the following notations will be utilized below.
The symbol $\mkbld\phi$ will denote the $N\times MN$ matrix
$(\phi_1\ \cdots\ \phi_{MN})$ and $\killsphi{D}$ will be used as a
shorthand for the statement that $D(\phi_i)$ is equal to the zero
vector for each $1\leq i\leq MN$ (i.e. that the vector functions are in
the kernel of some linear operator $D$).  The $N\times N$ identity matrix will be written
simply as $I$ and more generally the $m\times m$ identity matrix for
any natural number $m$ will be denoted by $I_m$.

The main result can now be stated concisely as:
\begin{theorem}[Matrix Case]\label{thm:modo}
If the functions $\phi_i$ are chosen so that $\det \Phi\not=0$ then (a) the differential operator
\begin{equation}
K=I\,\partial^M-\left(\phi_1^{(M)}\ \cdots\ \phi_{MN}^{(M)}\right)\Phi^{-1}\left(\begin{matrix}
I\\
\partial I\\
\vdots\\
\partial^{M-1}I\end{matrix}\right)\label{eqn:QD}
\end{equation}
is the  unique monic\ MODO $K$ of order $M$ such that $\killsphi{K}$ and (b) if $L$ is any MODO such that $\killsphi{L}$ then there exists a MODO $Q$ such that $L=Q\circ K$.
\end{theorem}
 %
%\begin{remarks} Just to clarify, in \eqref{eqn:QD} $I$ denotes the $N\times N$ identity matrix, a matrix differential operator is called ``monic'' if $I$ is its leading coefficient, and a vector function $f$ is in the kernel of the MODO $L$ if $L(f)=0$ is the zero vector. If one more block row of derivatives and a block column of constant coefficient operators are appended to the matrix $\Phi$, then the formula \eqref{eqn:QD} can be equivalently defined as a quasi-determinant or a Schur complement of that enlarged matrix, and in the case that $N=1$ it agrees with \eqref{eqn:WrK}.
%\end{remarks}

\begin{proof} 
For a natural number $m\geq M$, let $\mkbld{\Phi}_m$ denote the $(mN+N)\times (MN+N)$ matrix with block decomposition
$$
\mkbld{\Phi}_{m}=\left(\begin{matrix} \phi_1&\phi_2&\cdots&\phi_{MN}&I\\
 \phi_1'&\phi_2'&\cdots&\phi_{MN}'&\partial I\\
 \vdots&\vdots&\cdots&\vdots&\vdots\\
  \phi_1^{(m)}&\phi_2^{(m)}&\cdots&\phi_{MN}^{(m)}&\partial^{m} I
  \end{matrix}\right).
$$
It is a nearly trivial observation that for a MODO $L$ of order at most $m$
one has
\begin{equation}
L=\sum_{i=0}^m\alpha_i(x)\partial^i\ 
\Leftrightarrow
\ 
\left(\alpha_0\ \alpha_1\ \cdots\ \alpha_m\right)\mkbld{\Phi}_m=(L(\phi_1)\ \cdots\ L(\phi_{MN})\ L).
\label{eqn:alpha}
\end{equation}
That is, for any choice of $L$ the product of the $N\times (mN+N)$ matrix $\mkbld\alpha$
made from its coefficients 
with $\mkbld{\Phi}_m$ has the vector  $L(\phi_i)$ as its $i^{th}$ column for $1\leq i\leq MN$ 
and the last $N\times
N$ block is a copy of the operator itself, and conversely for 
any choice of $N\times (mN+N)$ matrix $\mkbld\alpha$
its product with $\mkbld{\Phi}_m$ yields a matrix that records a differential operator $L$ and its action on each of the vector functions as in
\eqref{eqn:alpha}.

In the case that $m=M$ and the $N\times N$ blocks $\alpha_i$ are defined by
$$
\mkbld\alpha=(\alpha_0\ \cdots\ \alpha_M)=\left(-\phi_1^{(M)} \cdots\ -\phi_{MN}^{(M)}\ I\right)\left(\begin{matrix}\Phi^{-1}&0\\
0&I\end{matrix}\right),
$$
the product  $\mkbld\alpha\mkbld{\Phi}_m$ in \eqref{eqn:alpha} would equal
$$
\left(-\phi_1^{(M)} \cdots\ -\phi_{MN}^{(M)}\ I\right)\left(\begin{matrix}\Phi^{-1}&0\\
0&I\end{matrix}\right)\mkbld{\Phi}_M=\left(-\mkbld\phi^{(M)}\ I\right)\left(\begin{matrix}
I_{MN}&B\\
\mkbld\phi^{(M)}&\partial^M I\end{matrix}\right)=(0\ K)
$$ for some MODO $K$.  Because the MODOs in the block $B$ are of
degree at most $M-1$, $K$ is monic of degree $M$.  We know that $\killsphi{K}$ 
since the first 
$MN$ columns, which are all zero, record its action on the functions
$\phi_i$.  This operator $K$ can
equivalently be produced with the formula \eqref{eqn:QD}, as the
quasi-determinant $|\mkbld{\Phi}_M|_{M+1,M+1}$ where $\Phi_M$ is viewed as an $(M+1)\times(M+1)$ matrix with entries that are $N\times N$ matrices or as the Schur complement of
the invertible block $\Phi$ in the matrix $\mkbld{\Phi}_M$.  This demonstrates
the \textit{existence} of the operator $K$ promised in (a).  (This could also have been achieved in many other ways, including merely by direct computation of $K(\phi_i)$, but doing it this way sets us up nicely for proving the rest of the claim.)

%% This shows the existence of the operator $K$.  Its uniqueness and most importantly the fact that any other operator $L$ satisfying $\killsphi{L}$ has a right factor of $K$ remain to be shown.

Let $G$ be the $(mN+N)\times(mN+N)$ matrix whose decomposition into $N\times (mN+N)$ blocks
$$
G=\left(\begin{matrix}G_0\\G_1\\\vdots\\G_{m}\end{matrix}\right)
\qquad
\hbox{is such that}
\qquad
\left(\begin{matrix}G_0\\G_1\\\vdots\\G_{M-1}\end{matrix}\right)=\left(\begin{matrix}\Phi^{-1}&0\end{matrix}\right)
$$
and for $i\geq M$ the block row
$\displaystyle
G_i=( g_{i0} \ g_{i1}\ \cdots\ g_{im} )$ where the functions $g_{ij}$ are defined by the equation
$\partial^{i-M}\circ K=\sum_{j=0}^{m} g_{ij}(x)\partial^j$.
Consider the product $G\mkbld{\Phi}_m$.  Its first $MN$ rows would have the form $(I_{MN}\ B)$ since $\Phi^{-1}\Phi=I_{MN}$.  For $i\geq M$, the product of block row $G_i$ with $\mkbld{\Phi}_m$ is designed so that the MODO $\partial^{i-M}\circ K$ shows up as its last $N\times N$ block.  According to \eqref{eqn:alpha}, the previous columns would be the image of $\mkbld\phi$ under the action of that operator, but $\killsphi{\partial^{i-M}\circ K}$ so
\begin{equation}
G\mkbld{\Phi}_m=\left(\begin{matrix}I_{MN}&B\\
0&K\\
0&\partial\circ K\\
\vdots&\vdots\\
0&\partial^{m-M}\circ K\end{matrix}\right).\label{eqn:GPhim}
\end{equation}

Now, suppose $L$ is any MODO of order at most $m$ satisfying $\killsphi{L}$.
Label its coefficients 
as in \eqref{eqn:alpha}
and consider the product of the $N\times (mN+N)$ matrix\footnote{
The matrix $G$ is invertible because its top-left $MN\times MN$ block
($\Phi^{-1}$) is invertible and below that it is lower-triangular with $1$'s along the diagonal.  }
$\mkbld{q}=
(\alpha_0\ \cdots\ \alpha_m)G^{-1}$ and the matrix $\mkbld\Psi=G\mkbld{\Phi}_m$.  On the one hand, because
the $G^{-1}$ and $G$ cancel, $\mkbld{q}\mkbld\Psi=(\alpha_0\ \cdots\ \alpha_m)\mkbld{\Phi}_m$ is equal to the expression on the
right in \eqref{eqn:alpha}.  Specifically, given the assumption
$\killsphi{L}$ it has the form $(0\ \cdots\ 0\ L)$.  

On the other hand, defining the $N\times N$ functions $q_i$ by the block decomposition
$\mkbld{q}=(q_0\ \cdots\ q_m)$ and making use
of the block
form of $\mkbld\Psi=G\mkbld{\Phi}_m$ in \eqref{eqn:GPhim} it is clear that $\mkbld{q}\mkbld\Psi$ is \textit{also}
equal to $( q_0\ q_1\ \ldots q_{M-1}\ L)$, because the block $I_{MN}$ picks out and preserves the precisely the first $M$ coefficient blocks.  

Combining these two observations, we conclude that $q_i=0$ for $0\leq i\leq M-1$.
Considering only the last block column in the product $\mkbld{q}\mkbld\Psi$, one obtains an
expression for $L$ as a sum involving the functions $q_i$ as coefficients multplied by the operators in the last block column of $G\mkbld{\Phi}_m$, 
but given the fact that the first $M$ of those coefficients are zero this
simplifies to:
$$
L=\sum_{i=M}^{m} q_i \left(\partial^{i-M}\circ K \right)
=\left(\sum_{i=0}^{m-M} q_{i+M}\partial^{i}\right)\circ K.
$$
Then the operator in parentheses above is the operator $Q$ satisfying
the claim in (b).

Finally, suppose  $\hat K$ was also a  monic MODOs of order $M$ such that $\killsphi{\hat K}$.  Then $D=K-\hat K$ is an operator of order strictly less than $M$ with this same property.  The only way that $D$ can have order less than $M$ and also satisfy $D=Q\circ K$ for some MODO $Q$ is if $Q=D=0$, which demonstrates the uniqueness of $K$ and completes the proof.
\end{proof}

\section{Example}

Consider the case $M=N=2$ and the vector functions
$$
\phi_1=\left(\begin{matrix}x^3\cr -x^3\end{matrix}\right),
\ 
\phi_2=\left(\begin{matrix}x^2\cr0\end{matrix}\right),
\ 
\phi_3=\left(\begin{matrix}0\cr x\end{matrix}\right),
\ \hbox{and}\ 
\phi_4=\left(\begin{matrix}1\cr0\end{matrix}\right).
$$
Then the block Wronskian matrix in \eqref{eqn:blockWr} is the matrix
$$
\Phi=\left(\begin{matrix}
 x^3 & x^2 & 0 & 1 \cr
 -x^3 & 0 & x & 0 \cr
 3 x^2 & 2 x & 0 & 0 \cr
 -3 x^2 & 0 & 1 & 0 \cr
 \end{matrix}\right),
 $$
 whose invertibility assures us that there is a \textit{unique} monic MODO of order $2$ having all four of these vectors in its kernel. Using \eqref{eqn:QD} we can easily determine that this operator is
 $$
 K=I \,\partial^2+\left(\begin{matrix}\strut -\frac{1}{x} & \frac{3}{2 x} \cr
 0 &\strut -\frac{3}{x}\end{matrix}\right)\partial
 +\left(\begin{matrix} \strut 0 & -\frac{3}{2 x^2} \cr
 \strut 0 & \frac{3}{x^2} \end{matrix}\right).
 $$
 
  Another MODO that  obviously also has each $\phi_i$ in its kernel is $$
 L=\left(\begin{matrix}1&1\cr1&1\end{matrix}\right)\partial^3.
 $$
 Especially since the leading coefficient of $L$ is non-zero and singular, a situation that cannot arise in the case $N=1$, without the theorem above it would not be clear that there is an algebraic relationship between $L$ and $K$.
  However, Theorem~\ref{thm:modo} assures us that \textit{any} MODO having these vectors in its kernel must have $K$ as a right factor. and so there must be a differential operator $Q$ such that $L=Q\circ K$.   (In fact, one can check that
  $$
  Q= \left(\begin{matrix}1&1\cr1&1\end{matrix}\right)\partial + \left(\begin{matrix}\strut \frac{1}{x} & \frac{3}{2 x} \cr
 \strut \frac{1}{x} & \frac{3}{2 x}\end{matrix}\right)
 $$
 realizes this factorization.)

 \begin{acknowledgment}%{Acknowledgment.}
 The author wishes to thank Maarten Bergvelt, Michael Gekhtman, Tom Kunkle and
 Chunxia Li for advice, assistance and encouragement.
 \end{acknowledgment}

\end{document}